\documentclass[reqno, 10pt]{amsart}

\usepackage[english]{babel}
\usepackage[utf8]{inputenc} % input of accented letters										

\usepackage{comment} % Comment environment
\usepackage{latexsym} %Includes some math symbols
\usepackage{amssymb,amsmath,amsfonts,amscd} %Enhance eqns, thm enviroments, ...
\usepackage{amsthm} %For theorems
\usepackage[mathscr]{euscript} %Alternative fonts

\theoremstyle{plain}
\newtheorem{theorem}{Theorem}[section]%[chapter]
\newtheorem{alfthm}{Theorem}[section]

\newtheorem{lemma}[theorem]{Lemma}

\newtheorem{proposition}[theorem]{Proposition}

\newtheorem{remark}[theorem]{Remark}

\newtheorem{definition}{Definition}[section]
\newtheorem{conjecture}{Conjecture}[section]

\DeclareMathOperator{\R}{\mathbb{R}}

\DeclareMathOperator{\N}{\mathbb{N}}
\DeclareMathOperator{\Z}{\mathbb{Z}}

\def\P{\mathbb{P}^{1}}

\def\SL{SL(2,\R)}
\def\GL{GL(2,\R)}
\def\id{\operatorname{id}}

\def\hp{\hat{p}}
\def\tq{\tilde{q}}
\def\quand{\quad\text{and}\quad}

\begin{document}

\title
[Probability distributions with non-compact support]
{Lyapunov exponents of probability distributions with non-compact support}

\author{Adriana S\'anchez}
%\address{ICMC-USP -- Avenida Trabalhador S\~{a}o-carlense 400, Centro, 13566-590 S\~{a}o Carlos S\~{a}o Paulo, Brazil.}
\address{ICMC-USP -- Av. Trab. S\~{a}o-carlense 400, S\~{a}o Carlos, 13566-590 S\~{a}o Paulo, Brazil.}
\email{asanchez@icmc.usp.br}

\author{Marcelo Viana}
\address{IMPA -- Estrada D. Castorina 110, Jardim Bot\^anico, 22460-320 Rio de Janeiro, Brazil.}
\email{viana@impa.br}

\thanks{Work was partially supported by Fondation Louis D. – Institut de France (project coordinated by M. Viana), CNPq and FAPERJ.
A.S. was supported by FAPESP (Fundação de Amparo à Pesquisa do Estado de São Paulo), grant 2018/18990-0 and Universidad de Costa Rica.}

\date{\today}

\keywords{Lyapunov exponents, linear cocycles, Wasserstein topology}
\subjclass[2010]{Primary: 37H15; Secondary: 37A20, 37D25}

%------------------------------------------------------------------------------------------------------
%Abstract

\begin{abstract}
A recent result of Bocker--Viana asserts that the Lyapunov exponents of compactly supported probability distributions in $\GL$ depend continuously on the distribution. We investigate the general, possibly non-compact case. We prove that the Lyapunov exponents are semi-continuous with respect to the Wasserstein topology, but not with respect to the weak* topology. Moreover, they are not continuous with respect to the Wasserstein topology.
\end{abstract}

%------------------------------------------------------------------------------------------------------

\maketitle

%------------------------------------------------------------------------------------------------------
%Intro

\section{Introduction}

Let $M=\SL^{\Z}$ and $f:M\to M$  be the shift map over $M$ defined by
\[
 (\alpha_n)_n\mapsto (\alpha_{n+1})_n.
\]
Consider the function
\[
 A:M\to\SL,\quad(\alpha_n)_n\mapsto \alpha_0,
\]
and define its $n$-th iterate by
\[
 A^n((\alpha_k)_k)=\alpha_{n-1}\cdots\alpha_0.
\]

Given any probability measure $p$ in $\SL$, the associated Bernoulli measure $\mu=p^{\Z}$ on $M$ is invariant under $f$.
Let $L^1(\mu)$ denote the space of $\mu$-integrable functions on $M$.
Assuming that $\log^+\|A^{\pm}\|\in L^1(\mu)$, it follows from Furstenberg-Kesten~\cite[Theorem~2]{FurstenbergKesten} that
\[
 \lambda_+(x)=\lim_{n}\frac{1}{n}\log\|A^n(x)\|\quand\lambda_-(x)=\lim_{n}\frac{1}{n}\log\|A^{-n}(x)\|^{-1},
\]
exist for $\mu$-almost every $x\in M$. Furthermore, the functions $x \mapsto \lambda_\pm(x)$ are constant on the orbits of $f$.
Thus, by the ergodicity of the Bernoulli measure, they are constant on a full $\mu$-measure set.
We call these constants the \emph{Lyapunov exponents} of $\mu$, and we represent them as $\lambda_\pm(p)$.

%The Lyapunov exponents measure the average exponential growth of the norm iterates of the cocycle along invariant subspaces in the fibers.
%They describe the chaotic behavior of the system. For example, a strictly positive maximal Lyapunov exponent is synonymous of
%exponential instability. It is an indication that the system modeled by the cocycle behaves chaotically,
%and the maximal Lyapunov exponent measures the chaos. The continuity can be interpreted as the persistence of chaos under small perturbations.

The way Lyapunov exponents depend on the corresponding probability distribution $p$ has been studied by several authors.
Furstenberg, Kifer~\cite{FurstenbergKifer} proved that this dependence is continuous at every quasi-irreducible point, that is,
such that there is at most one subspace invariant under every matrix in the support of $p$.
Also under an irreducibility condition, LePage~\cite{LePage} proved H\"older continuity of the Lyapunov exponents.

More recently, Bocker and Viana \cite{BockerViana} proved that continuity actually holds at every point,
in the realm of compactly supported probability distributions. Backes, Butler and Brown \cite{BBB} have much extended this result,
to linear cocycles with invariant holonomies, over hyperbolic systems with local product structure on compact spaces.
Moreover, still in the compactly supported case, Tall and Viana~\cite{VianaYaya} proved that H\"older continuity holds whenever
the exponents $\lambda_-(p)$ and $\lambda_+(p)$ are distinct. In general, when the exponents coincide, one has a weaker but still
quantitative, modulus of continuity, called log-H\"older continuity.

For much more on this topic, see Viana~\cite[Chapter~10]{LLE}, Duarte, Klein~\cite{DuK}, Viana~\cite{Disc} and the references therein.

The aim of this note is to study the (semi)continuity of the Lyapunov exponents in the general case, when the support of the probability
measure $p$ is not necessarily compact. Our main result reads as follows
(Theorem \ref{th.NotSemicontWeakTop} and Theorem \ref{th.SemicontWass} provide more detailed statements):

\begin{alfthm}\label{th.TheoremA}
The function $p\mapsto\lambda_{+}(p)$ is upper semi-continuous for the Wasserstein topology but not for the weak* topology.
The same remains valid for $p\mapsto\lambda_{-}(p)$ with lower semi-continuity.
\end{alfthm}

Regarding continuity of Lyapunov exponents we prove (see Theorem \ref{th.NotContWass} for a more detailed statement):

\begin{alfthm}\label{th.TheoremB}
The functions $p\mapsto\lambda_\pm(p)$ are not continuous for the Wasserstein topology.
\end{alfthm}

\medskip{\bf Acknowledgements.} We thank Mauricio Poletti and El Hadji Yaya Tall for many helpful comments and suggestions.

%Wasserstein
\section{Wasserstein Topology}

Let $(M,d)$ be a Polish space, that is a complete separable metric space.
The \emph{Wasserstein space} is the space of probability measures on $M$ with finite moments of order $1$,
that is,
\[
P_1(M)=\left\lbrace \mu\in P(M):\int_M d(x_0,x)d\mu(x)<+\infty\right\rbrace,
\]
where $x_0\in M$ is arbitrary and $P(M)$ denotes the space of Borel probability measures on $M$.
This space does not depend on the choice of the point $x_0$.

We use $\mu_k\overset{\ast}{\longrightarrow}\mu$ to mean convergence in the weak* topology.
Moreover, $\mu_k\overset{W}{\longrightarrow}\mu$ will mean convergence in the Wasserstein topology,
characterized by the following definition (see Definition~6.8 in Villani~\cite{Villani}):

\begin{definition}\label{def.Wconv}
Let $(M,d)$ be a Polish space and $(\mu_k)_{k\in \N}$ be a sequence in $P_1(M)$.
Given any $\mu \in P_1(M)$, we say that $\mu_k\overset{W}{\longrightarrow}\mu$ if one of the following
equivalent conditions is satisfied for some (and then any) $x_0\in M$:
\begin{enumerate}
%\item $\mu_k\overset{W}{\longrightarrow}\mu$;
\item $\mu_k\overset{\ast}{\longrightarrow}\mu$ and $\int d(x_0,x)d\mu_k(x)\to\int d(x_0,x)d\mu(x);$
\item $\mu_k\overset{\ast}{\longrightarrow}\mu$ and
      \[
       \limsup_{k\to\infty}\int d(x_0,x)d\mu_k(x)\leq \int d(x_0,x)d\mu(x);
      \]
\item $\mu_k\overset{\ast}{\longrightarrow}\mu$ and
      \[
       \lim_{R\to\infty}\limsup_{k\to\infty}\int_{d(x_0,x)\geq R} d(x_0,x)d\mu_k(x)=0;
      \]
\item \label{Item.4PropWconv} For every continuous function $\varphi$ with $|\varphi(x)|\leq C(1+d(x_0,x))$, $C\in\R$, one has
      \[
       \int\varphi(x)d\mu_k(x)\to\int\varphi(x)d\mu(x).
      \]
\end{enumerate}
\end{definition}

Let $(M,\mu)$ and $(N,\nu)$ be two probability spaces. A \emph{coupling} of $\mu$ and $\nu$ is a measure $\pi$ on $M\times N$
such that $\pi$ projects to $\mu$ and $\nu$ on the first and second coordinate, respectively.
When $\mu=\nu$ we call $\pi$ a \emph{self-coupling}.

If $(M,d)$ is a Polish space, the \emph{Wasserstein distance} between two probability measures $\mu$ and $\nu$ on $M$
is defined by

\begin{equation}\label{eqn.WassersteinDist}
 W_1(\mu,\nu)=\inf_{\pi\in\Pi(\mu,\nu)}\int_M d(x,y)d\pi(x,y),
\end{equation}
where the infimum is taken over the set $\Pi(\mu,\nu)$ of all the couplings of $\mu$ and $\nu$.
This distance metrizes the Wasserstein topology (see \cite[Theorem ~6.18]{Villani}):

\begin{theorem}\label{th.WarTop}
Let $(M,d)$ be a Polish space. Then the Wasserstein  distance $W_1$ metrizes the Wasserstein topology
in the space $P_1(M)$:
$$
\mu_k\overset{W}{\longrightarrow}\mu \text{ if and only if } W_1(\mu_k,\mu)\to 0.
$$
Moreover, with this metric $P_1(M)$ is also a complete separable metric space and, every $\mu\in P_1(M)$
can be approximated by probability measures with finite support.
\end{theorem}

This also implies that $W_1$ is continuous on $P_1(M)$.

%------------------------------------------------------------------------------------------------------
%Semicontinuity

\section{Semicontinuity}\label{sec.Sem}

It is a well-known fact that when the measures have compact support, the Lyapunov exponents are semicontinuous with the weak* topology (see for example \cite[Chapter~9]{LLE}). However, in the non compact setting this is no longer true. If they were semicontinuous then every measure with vanishing Lyapunov exponent would be a point of continuity. The next theorem shows that this is not the case.

\begin{theorem}\label{th.NotSemicontWeakTop}
There exist a measure $q$ and a sequence of measures $(q_n)_n$ on $\SL$ converging to $q$ in the weak* topology, such that $\lambda_+(q_n)\geq1$ for $n$ large enough but $\lambda_{+}(q)=0$.
\end{theorem}

\begin{proof}
Define a function $\alpha:\N\to\SL$ by
\begin{equation*}
\alpha(2k-1) = \begin{pmatrix}
                  \sigma_k & 0\\
                  0 & \sigma_k^{-1}\\
                 \end{pmatrix}
\quand
\alpha(2k) = \begin{pmatrix}
                  \sigma_k^{-1} & 0\\
                  0 & \sigma_k\\
               \end{pmatrix} \\
\end{equation*}
where $(\sigma_k)_k$ is an increasing sequence such that $\sigma_1>1$ and $\sigma_k\to+\infty$.

Let $\mu=q^{\Z}$ be a measure in $M$ where $q$ is the measure on $\SL$ given by
\[
 q=\sum_{k\in \N} p_k\delta_{\alpha(k)},
\]
with $\sum p_k=1$, $0<p_k<1$ for all $k\in\N$.

The key idea to construct this example is to find $p_k$ and $\sigma_k$ such that $\log\|A\|\in L^1(\mu)$ and satisfying the hypothesis above. Hence, consider $0<r<1/2<s<1$, and $l=s/r>1$. Let us take $\sigma_k=e^{l^k}$ for all $k$, which is an increasing sequence provided that $l>1$.

For $k\geq 2$ take $p_{2k-1}=p_{2k}=r^k$. Since $0<r<1/2 $ it is easy to see that
\[
 \sum_{k\geq 3}p_k = 2\sum_{k\geq 2}p_{2k}= 2\sum_{k\geq 2}r^k= 2\frac{r^2}{1-r}<1
\]
We have to choose $p_1$ and $p_2$ such that $\sum p_k=1$. Then, it is enough to take
\[
 p_1=p_2=\frac{1}{2}\left(1-2\frac{r^2}{1-r}\right).
\]

We continue by showing that  $\log\|A\|\in L^1(\mu)$. This is an easy computation,
\[
 \int_{M}\log\|A\|d\mu = 2p_2\log\sigma_1+2\sum_{k\geq 2}p_{2k}\log\sigma_k= 2p_2l+2\sum_{k\geq 2}s^k.
\]
Since $0<s<1$ this geometric series is convergent. More over, since $p_{2k-1}=p_{2k}$ for all $k$ then $\lambda_{+}(q)=0$.

What is left is to construct the sequence $q_n$. Fix $n_0>1$ large enough so $\frac{1}{2}\left(1-2\frac{r^2}{1-r}\right)>l^{-n}$ for all $n\geq n_0$, and consider $q_n=\sum_kq^n_k\delta_{\alpha(k)}$ where for $n\geq n_0$
\begin{eqnarray*}
q^n_{2n} &=& l^{-n}+r^n,\\
%q^n_1 &=& \frac{1}{2}\left(1-\frac{r^2}{1-r}\right),\\
q^n_2 &=& \frac{1}{2}\left(1-\frac{r^2}{1-r}\right)-l^{-n}\\
q^n_k &=& p_k\text{ other case.}
\end{eqnarray*}
Thus, since $q^n_k\to p_k$ when $n\to\infty$ for all $k$, it is easy to see that $q_n$ converges in the weak* topology to $q$.

The proof is completed by showing that $\lambda_+(q_n)\geq 1$ for $n$ large enough. It follows easily since,
\[
 \lambda_+(q_n)=|q^n_{2n-1}-q^n_{2n}|\log\sigma_n+|q^n_1-q^n_2|\log\sigma_1=l^{-n}l^n+l^{-n+1},
\]
which is equal to $1+l^{-n+1}\geq 1$ for all $n\geq n_0$.
\end{proof}

We now consider the Wasserstein topology in $P_1(\SL)$ which is stronger than the weak* topology,
as stated in Definition~\ref{def.Wconv}.
The advantage of using this topology is that all probability measures in $P_1(\SL)$ have finite moment of order 1,
and so their Lyapunov exponents always exist.
This observation is a direct consequence of the fact that $\log:[1,\infty)\to\R$ is a 1-Lipschitz function and
$\|\alpha\|\geq 1$ for every matrix $\alpha\in\SL$, because
\[
 \int\log\|A(x)\|d\mu=\int\log\|\alpha\|dp\leq \int d(\alpha,\id)dp<\infty.
\]
The convergence of the moments of order 1, allow us to control the weight of integrals outside compact sets,
and prove semicontinuity of the Lyapunov exponents in $P_1(\SL)$. We are thus led to the following result:

\begin{theorem}\label{th.SemicontWass}
The function defined on $P_1(\SL)$ by $p\to\lambda_{+}(p)$ is upper semi-continuous. The same remains valid for the function $p\to\lambda_-(p)$ with lower semi-continuity.
\end{theorem}

Before beginning the proof of Theorem \ref{th.SemicontWass} we need to recalled some important results regarding the relationship between Lyapunov exponents and stationary measures.

A probability measure $\eta$ on $\P$ is called a p-\emph{stationary} if
\[
 \eta(E)=\int\eta(\alpha^{-1}E)dp(\alpha),
\]
for every measurable set $E\in\P$ and $\alpha^{-1}E=\{[\alpha^{-1}v]:[v]\in E\}$.

Roughly speaking, the following result shows that the set of stationary measures for a measure $p$ is close for the weak* topology.

\begin{proposition}\label{prop.ClosenessStat}
Let $(p_k)_k$ be probability measures in $\SL$ converging to $p$ in the weak* topology. For each $k$, let $\eta_k$ be $p_k$-stationary measures and $\eta_k$ converges to $\eta$ in the weak* topology. Then $\eta$ is a stationary measure for $p$.
\end{proposition}

Furthermore, in our context it is well-known that
\[
 \lambda_+(p)=\max\left\{\int\Phi dp\times\eta:\eta\hspace{2mm}p-\text{stationary}\right\},
\]
where $\Phi:\SL\times\P\to\R$ is given by
\[
 \Phi(\alpha,[v])=\log\frac{\|\alpha v\|}{\|v\|}.
\]
For more details see for example \cite[Proposition ~6.7]{LLE}.

We now proceed to the proof of Theorem \ref{th.SemicontWass}.

\begin{proof}[Proof of Theorem \ref{th.SemicontWass}]
We will prove that $\lambda_+(p)$ is upper semi-continuous. The case of $\lambda_-(p)$ is analogous.

Let $(p_k)_k$ be a sequence in the Wasserstein space $P_1(M)$ converging to $p$, i.e. $W(p_k,p)\to0$. For each $k\in\N$ let $\eta_k$ a stationary measure that realizes the maximum in the identity above. That is:
\[
 \lambda_+(p_k)=\int\Phi dp_k d\eta_k.
\]
Since $\P$ is compact, passing to a subsequence if necessary, we can suppose $\eta_k$ converges in the weak* topology to a measure $\eta$ which, as established in Proposition \ref{prop.ClosenessStat}, is a $p$-stationary measure.

Let $\epsilon>0$, we want to prove that there exist a constant $k_0\in\N$ such that for each $k>k_0$
\[
 \left|\int\Phi dp_k d\eta_k-\int\Phi dp d\eta\right|<\epsilon.
\]

In order to do this we need to consider some properties of the Wasserstein topology. First of all, since the first moment of $p$ is finite there exist $K_1$ a compact set of $\SL$ such that
\begin{equation}\label{eqn.FirstMoment}
 \int_{K_1^c}d(\alpha,\id)dp <\frac{\epsilon}{36}.
\end{equation}

Moreover, according to Definition~\ref{def.Wconv}, since $p_k\overset{W}{\longrightarrow}p$ there exist $R'>0$ satisfying
\[
 \limsup_k \int_{d(\alpha,\id)>R'}d(\alpha,\id)dp_k<\frac{\epsilon}{36},
\]
then, there exist $k'>0$ such that  for every $k>k'$
\begin{equation}\label{eqn.Limsup}
 \int_{d(\alpha,\id)>R'}d(\alpha,\id)dp_k<\frac{\epsilon}{36}.
\end{equation}
Take $R>0$ big enough so $B(\id,R')\cup K_1\subset B(\id,R)$ and define the compact set $K=\bar{B}(\id,R)$.
Since the function $\log:[1,\infty)\to\R$ is 1-Lipschitz and $\|\alpha\|\geq1$ for all $\alpha\in\SL$, then
\begin{equation}
 |\Phi(\alpha,[v])|=\left|\log\frac{\|\alpha v\|}{\|v\|}\right|\leq \log\|\alpha\|\leq |\|\alpha\|-\|\id\||\leq d(\alpha,\id).
\end{equation}

Our proof starts with the observation that
\begin{align*}
 &\left|\int\Phi dp_k d\eta_k-\int\Phi dp d\eta\right|\\
 &\leq \left|\int_{K\times\P}\Phi dp_k d\eta_k-\int_{K\times\P}\Phi dp d\eta\right|+\left|\int_{K^c\times\P}\Phi dp_k d\eta_k\right|+\left|\int_{K^c\times\P}\Phi dp d\eta\right|.\\
\end{align*}
From \eqref{eqn.Limsup} it follows that
\begin{equation}\label{eqn.Integral1}
 \left|\int_{K^c\times\P}\Phi dp_k d\eta_k\right|\leq\int_{K^c}d(\alpha,\id) dp_k<\frac{\epsilon}{3}.
\end{equation}
Furthermore, \eqref{eqn.FirstMoment} implies that
\begin{equation}\label{eqn.Integral2}
 \left|\int_{K^c\times\P}\Phi dp d\eta\right|\leq\int_{K^c}d(\alpha,\id) dp<\frac{\epsilon}{3}.
\end{equation}

We now proceed to analyze the integral:
\begin{align*}
 &\left|\int_{K\times\P}\Phi dp_k d\eta_k-\int_{K\times\P}\Phi dp d\eta\right|\\
 &\leq \left|\int_{K\times\P}\Phi dp_k d\eta_k-\int_{K\times\P}\Phi dp_k d\eta\right|+\left|\int_{K\times\P}\Phi dp_k d\eta-\int_{K\times\P}\Phi dp d\eta\right|.
\end{align*}
Consider $\Phi_K=\Phi|_{K\times\P}$ the restriction of $\Phi$ to the compact space $K\times\P$. Thus, $\Phi_K$ is uniformly continuous with the product metric. Hence, there exist $\delta=\delta(\epsilon)$ such that for every $[v]\in\P$ and every $\alpha,\beta\in K$ satisfying $d(\alpha,\beta)<\delta$ we have
\[
 \left|\Phi_K(\alpha,[v])-\Phi_K(\beta,[v])\right|<\frac{\epsilon}{18}.
\]

Moreover, by the compactness of the set $K$ we can find $\alpha_1,...,\alpha_N\in K$ such that $K\subset\cup_{i=1}^N B(\alpha_i,\delta)$. Therefore, the convergence of $(\eta_k)_k$ to $\eta$ in the weak* topology implies that for each $i=1,...,N$ there exist $k_i>0$ such that
\[
 \left|\int_{\P}\Phi_K(\alpha_i,[v]) d\eta_k-\int_{\P}\Phi_K(\alpha_i,[v])d\eta\right|<\frac{\epsilon}{18}.
\]
Take $k''=\max\{k_1,...,k_N\}$. From the above it follows that given $\alpha\in K$ there exist $i$ such that $d(\alpha,\alpha_i)<\delta$ and for every $k>k''$ if $\Phi_i=\Phi_K(\alpha_i,[v])$ then
\[
\begin{aligned}
& \left|\int_{\P}\Phi_K d\eta_k-\int_{\P}\Phi_K d\eta\right| \le \\
& \hspace*{1cm} \leq \int_{\P}\left|\Phi_K -\Phi_i\right| d\eta_k +
\left|\int_{\P}\Phi_i d\eta_k-\int_{\P}\Phi_i d\eta\right|+\int_{\P}\left|\Phi_i-\Phi_K\right|d\eta<\frac{\epsilon}{6}.
\end{aligned}
\]
Since this convergence is uniform in $\alpha$, this implies that
\begin{equation}\label{eqn.Integral31}
 \left|\int_{K^c\times\P}\Phi_K(\alpha,[v]) d\eta_k dp_k-\int_{K^c\times\P}\Phi_K(\alpha,[v])d\eta dp_k\right|<\frac{\epsilon}{6},
\end{equation}
for all $k>k''$.

Now define $L=\SL\backslash B(id,R+1)$ and, consider the Urysohn function $f_0:\SL\to[0,1]$ given by
\[
 f_0(\alpha)=\frac{d(\alpha,L)}{d(\alpha,L)+d(\alpha,K)}.
\]
It is easily seen that $f_0$ is continuous, equal to zero in $L$ and equal to 1 in $K$. Therefore, the functions
\[
 \psi(\alpha)=\int_{\P}\Phi(\alpha,[v])d\eta\cdot f_0(\alpha)
\]
are continuous. Then since $|\psi(\alpha)|\leq d(\alpha,id)$ and, by the definition of the Wasserstein topology, there exist $k'''$ such that for every $k>k'''$
\[
 \left|\int\psi dp_k-\int\psi dp\right|<\frac{\epsilon}{18}.
\]
Moreover, since $\psi-\varphi=0$ in $K$ and
\[
 |\psi-\varphi|\leq\log\|\alpha\||f_0(\alpha)-\chi_K(\alpha)|\leq 2d(\alpha,id).
\]
Hence, by (\ref{eqn.FirstMoment}) and (\ref{eqn.Limsup})
\begin{align*}
 &\int|\psi-\varphi|dp\leq 2\int_{K^c}d(\alpha,id)dp<\frac{\epsilon}{18},\\
 &\int|\psi-\varphi|dp_k\leq 2\int_{K^c}d(\alpha,id)dp_k<\frac{\epsilon}{18}
\end{align*}
for each $k>k'$. Thus, if $k>\max\{k',k'''\}$ we get
\begin{equation}
 \left|\int_{K\times\P}\Phi dp_k d\eta-\int_{K\times\P}\Phi dp d\eta\right|<\frac{\epsilon}{6}.
\end{equation}
Finally, taking $k_0=\max\{k',k'',k'''\}$, we conclude that for every $k>k_0$
\[
 \left|\int\Phi dp_k d\eta_k-\int\Phi dp d\eta\right|<\epsilon.
\]
We just proved that
\[
 \lambda_+(p_k)=\int\Phi dp_k d\eta_k\to\int\Phi dp d\eta\leq \lambda_+(p),
\]
which concludes our proof.
\end{proof}

\begin{remark}
Theorem \ref{th.NotSemicontWeakTop} does not contradicts Theorem \ref{th.SemicontWass},
since $p$ is not in $P_1(\SL)$. To see this, take $x_0=\id$ and note that
\[
\int d(x,x_0)dp=\sum_{k=0}^{\infty}p_k\|\alpha_k-\id\|=2\sum_{k=0}^{\infty}r^k(e^{l^k}-1),
\]
which diverges.
\end{remark}

%------------------------------------------------------------------------------------------------------
%Continuity

\section{Proof of Theorem \ref{th.TheoremB}}

At this section we are going to describe a construction of points of discontinuity of the Lyapunov exponents as functions of the probability measure, relative to the Wasserstein topology.

\begin{theorem}\label{th.NotContWass}
There exist a measure $q$ and a sequence of measures $(q_n)_n$ on $\SL$ converging to $q$ in the Wasserstein topology, such that $\lambda_+(q_n)=0$ for all $n\in\N$ but $\lambda_{+}(q)>0$.
\end{theorem}

%Before beginning the proof of the Theorem \ref{th.NotContWass} we first need to remember the following result from Bougerol and Lacroix \cite[Theorem~6.1]{BougerolLacroix}.

%\begin{theorem}\label{th.Bougerol}
% If $\log^+(\|A\|)\in L^1(\mu)$ and the cocycle $F$ is irreducible then $\lambda_+(p)>\lambda_-(p)$ if and only if $F$ is strongly irreducible and contracting.
%\end{theorem}

%Having this result in mind, we now proceed to prove Theorem \ref{th.NotContWass}.
\begin{proof}

Consider the function $\alpha:\N\to\SL$ defined by the hyperbolic matrices
\[
\alpha(k) = \begin{pmatrix}
              k & 0\\
              0 & k^{-1}\\
            \end{pmatrix}
\]
Take $m\in\N$ the smallest natural number bigger than 1 such that $\sum_{n\geq m}e^{-\sqrt{n}}<1$, which exist since $\sum_k e^{-\sqrt{k}}$ is convergent, and define
\begin{align*}
 &p_k=e^{-\sqrt{k}},\text{ if }k\geq m,\\
 &p_1=1-\sum_{n\geq m}e^{-\sqrt{n}},\\
 &p_k=0,\text{ otherwise.}
\end{align*}

It is obvious from the definition that $\sum_k p_k=1$. Hence, we define the probability measure $q=\sum p_k\delta_{\alpha(k)}$. We need to see that $q\in P_1(\SL)$, in order to do so notice that if $x_0=\id$
\[
 \int d(x,x_0)dq=\sum_k p_k\|\alpha(k)-\id\|=\sum_ke^{-\sqrt{k}}(k-1)
\]
which is convergent by the Cauchy condensation test. Moreover, since
\[
 \sum_k e^{-\sqrt{k}}\log k<\sum_k ke^{-\sqrt{k}}
\]
then if $\mu=q^{\Z}$ we have $\log \|A\|\in L^1(\mu)$ and,
\[
 \lambda_+(q)=\sum_k e^{-\sqrt{k}}\log k>0.
\]
Now, consider $B=\begin{pmatrix} 0 & -1\\ 1 & 0\\ \end{pmatrix}$ and for each $n$ consider
\[
 \beta_n(k)=\left\lbrace\begin{array}{ll}
            \alpha(k) & \text{ if }k\neq n,\\
             B & \text{ if }k=n.\\
           \end{array}\right.
\]

With this we define the probability measures $q_n=\sum_k p_k\delta_{\beta_n(k)}$. In a similar way as above, we can see that for all $n$ these measures belong to $P_1(\SL)$ and, $\log \|A\|\in L^1(q_n)$. We proceed to show that $q_n$ converges to $q$ in the Wasserstein topology. This follows since
\[
 W(q_n,q)\leq p_nd(\alpha(n),B)\sim ne^{-\sqrt{n}}
\]
which goes to 0 if $n$ goes to $\infty$.

It remains to prove that $\lambda_{+}(q_n)=0$ for every $n$. In order to do this we proceed by contradiction. Suppose there exist $N$ such that $\lambda_{+}(q_N)>0>\lambda_{-}(q_N)$. We will consider the distance in the projective space $\P$ given by
\[
 \delta([v],[w]):=\frac{\|v\wedge w\|}{\|v\|\|w\|}=\sin(\angle(v,w)).
\]

Consider the family $\mathcal{F}=\{V,H\}$, where $V=[e_2]$ and $H=[e_1]$ are the vertical and horizontal axis respectively. By the definition of the measure $q_N$, it is clear that this family is invariant by every matrix in the support of this measure. Moreover, if $x=(x_k)_k\in M$ then we can see that for every $m$
\begin{equation}\label{eqn.HVdistance}
  \delta(A^m(x)H,A^m(x)V)\geq\delta(H,V)=1.
\end{equation}

On the other hand, we have for every unit vectors $v$ and $w$
\[
 \|A^m(x)v\wedge A^m(x)w\|\leq\|\wedge^2A^m(x)\|.
\]
It is widely known that, for $q_N$-almost every $x\in M$
\[
 \lambda_{+}(q_N)+\lambda_{-}(q_N)=\lim_{m}\frac{1}{m}\log\|\wedge^2A^m(x)\|.
\]

Note that $q_N$ is irreducible, this means that there is no proper subspace of $\R^2$ invariant under all the matrices in the support of $q_N$. Therefore, we have
\[
 \lambda_{+}(q_N)=\lim_{m}\frac{1}{m}\log\|A^m(x)w\|,
\]
for every unit vector $w$.
For a more detailed discussion of the two results mention above we refer the reader to \cite[Chapter~III]{BougerolLacroix}.

Thus, we have for every unit vectors $v$ and $w$
\[
 \lim_{m}\frac{1}{m}\log\frac{\|\wedge^2A^m(x)\|}{\|A^m(x)v\|\|A^m(x)w\|}=\lambda_{-}(q_N)-\lambda_{+}(q_N)<0,
\]
and hence
\begin{align*}
 \lim_{m}\delta(A^m(x)H,A^m(x)V)&\leq \lim_{m}\frac{\|\wedge^2A^m(x)\|}{\|A^m(x)e_1\|\|A^m(x)e_2\|}\\
   &=\lim_{m}\exp\left(m\cdot\frac{1}{m}\log \frac{\|\wedge^2A^m(x)\|}{\|A^m(x)e_1\|\|A^m(x)e_2\|}\right)=0,
\end{align*}

which is a contradiction with \eqref{eqn.HVdistance} and, we finish our proof.
%Therefore, as consequence of Theorem \ref{th.Bougerol} we have that $\lambda_+(q_n)=\lambda_-(q_n)=0$ for every $n$.

\end{proof}

\begin{comment}
\begin{remark}
 One can also consider a Kifer type example, that is making a weight goes to 0. For example, take
\begin{align*}
 &p=\sum_k p_k\delta_{\alpha(k)},\\
 &q_n=\sum q_k^n\delta_{\alpha(k)}+\frac{1}{n}\delta_B.
\end{align*}
where $q_k^n=\left(1-\frac{1}{n}\right)p_k$ and $p_k$ as before.

As before, $\lambda_+(p)>0$ and $\lambda_+(q_n)=0$ for every $n$. We are left with the task of proving the convergence in the Wasserstein topology. An easy computation shows that
\[
 W(q_n,p)\leq \frac{1}{n}\sum_k p_k d(\alpha(k),B)\sim \frac{1}{n}\sum_k ke^{-\sqrt{k}}
\]
which goes to 0 when $n$ goes to infinite.
\end{remark}
\end{comment}

Notice that this example shows that the Wasserstein topology is not enough to guarantee continuity of the Lyapunov exponents. The main problem is that the support of the measures $q_n$  move further apart from the support of $q$. Thus, this suggest that we need to add some hypothesis guaranteeing the ``convergence'' of the supports. An assumption of this type was made by Bocker, Viana in \cite{BockerViana} in order to prove the continuity for measures with compact support. However in the non compact setting we will also need to control the iteratives of the cocycles. Those hypotheses are very restrictive so we are left with the following conjecture:

\begin{conjecture}
 The Lyapunov exponent functions $\lambda_\pm:P_1(\SL)\to\R$ are not continuous with respect to the Wasserstein--Hausdorff topology.
\end{conjecture}

In the next two sections we are going to describe a construction of points of discontinuity of the Lyapunov exponents as functions of the measure, relative to the Wasserstein topology. However, in each of them the support of the measures are arbitrarily close. These constructions were inspired by the discontinuity example presented by Bocker Viana in \cite[Section~7.1]{BockerViana}, and do not prove our conjecture since they are not examples in $\SL$.

\subsection{Discontinuity example in $\SL^5$}\label{sec.DiscSL5}

Let us recall that $M=(\SL)^{\Z}$, $f:M\to M$  is the shift map over $M$ defined by
\[
 (\alpha_n)_n\mapsto (\alpha_{n+1})_n.
\]
And the linear cocycle $A$ is the product of random matrices which is defined by
\[
 A:M\to\SL,\quad(\alpha_n)_n\mapsto \alpha_0.
\]
Given an invariant measure $p$ in $\SL$ we can define $\mu=p^{\Z}$ which is an invariant measure in $M$.

Now consider $X=\SL^5$ with the product metric
\[
 d_{\infty}((\alpha_1,...,\alpha_5),(\beta_1,...,\beta_5))=\max\{d(\alpha_1,\beta_1),...,d(\alpha_5,\beta_5)\}.
\]
Let $N=X^{\Z}$ be the space of sequences over $X$ and $g:N\to N$ the shift map over $N$. We can identify $N$ with $M$ using the function $\iota:M\to N$ by $\iota((\alpha_n)_n)=(\beta_n)_n$ where
\[
 \beta_n=(\alpha_{5n},\alpha_{5n+1},\alpha_{5n+2},\alpha_{5n+3},\alpha_{5n+4})_n.
\]

It is easy to see that $\iota$ defines a bijection between $N$ and $M$. Moreover, we have the following identity
\[
 g(\iota((\alpha_n)_n))=f^5((\alpha_n)_n).
\]
Also we can consider the linear cocycle induced by $A$ in $N$, that is the function $B:N\to \SL$ given by
\[
 B((\iota((\alpha_n)_n))=A^5((\alpha_n)_n).
\]
So in this context we have the following result.

\begin{theorem}
There exist a measure $q$ and a sequence of measures $(q_n)_n$ on $X$ converging to $q$ in the Wasserstein topology, such that
\[
 \lambda_+(B,q_n)\nrightarrow\lambda_+(B,q).
\]
\end{theorem}
The main idea of the proof is to construct a measure on $N$ whose Lyapunov exponents are positive and approximate it, in the Wasserstein topology, by measures with zero Lyapunov exponents. In order to do that, define the function $\alpha:\N\to\SL$ as
%\begin{align*}
%\alpha(2k-1) &= \begin{pmatrix}
%                  k & 0\\
%                  0 & k^{-1}\\
%                 \end{pmatrix} \\\vspace{5mm}
%\empty & \empty \\
%\alpha(2k) &= \begin{pmatrix}
%                  k^{-1} & 0\\
%                  0 & k\\
%               \end{pmatrix}
%\end{align*}
\[
\alpha(2k-1) = \begin{pmatrix}
                  k & 0\\
                  0 & k^{-1}\\
                 \end{pmatrix}
\quand
\alpha(2k) = \begin{pmatrix}
                  k^{-1} & 0\\
                  0 & k\\
               \end{pmatrix}
\]
As in the example before take $m\in\N$ the smallest natural (odd) number bigger than 3 such that $\sum_{k\geq m}e^{-\sqrt{k}}<1$, which exist since $\sum_k e^{-\sqrt{k}}$ is convergent, and define
\begin{align*}
 &p_{2k}=p_{2k-1}=\frac{1}{2} e^{-\sqrt{k}},\text{ if }2k-1\geq m,\\
 &p_3=1-\sum_{n\geq m}e^{-\sqrt{n}},\\
 &p_k=0,\text{ otherwise.}
\end{align*}
Let $\mu=\tq^{\Z}$ be a measure in $M$ where $\tq$ is the measure on $\SL$ given by
\[
 \tq=\sum_{k\in \N}p_k\delta_{\alpha(k)}.
\]
Let us consider the space $\Omega=\N^5$ and define the measure on $X$ by
\[
 q=\sum_{w\in\Omega}p_{w}\delta_{\alpha(w)},
\]
where $\alpha(w)=(\alpha(w_1),\cdots,\alpha(w_5))$ and, $p_w=p_{w_1}\cdots p_{w_5}$ if $w=(w_1,...,w_5)$.

Now consider the measure $\nu=q^{\Z}$ on $N$. First, we need to ensure that the measure $q$ belong to $P_1(X)$. This is a direct consequence of the fact that $\sum e^{-\sqrt{n}}(n-1)$ is convergent equal to some positive constant $c$. Indeed, if $\alpha_0=(\id,...,\id)$ and the notation $p_1\cdots\hp_i\cdots p_5$ denotes the product of $p_1$ through $p_5$ except $p_i$ then
\begin{align*}
\int d_{\infty}(\alpha,\alpha_0)dq &=\sum_w p_wd_{\infty}(\alpha(w),\alpha_0)\\
    % &\leq \sum_{i=1}^5\sum_w p_{w_1}\cdots p_{w_5}d(w_i,\id)\\
     &<\sum_{i=1}^5\sum_{w_j,j\neq i} p_{w_1}\cdots\hp_{w_i}\cdots p_{w_5}\left(\sum_{w_i}p_{w_i}d(\alpha(w_i),\id)\right)\\
     &<c\sum_{i=1}^5\sum_{w_j,j\neq i} p_{w_1}\cdots\hp_{w_i}\cdots p_{w_5}
     =5c
\end{align*}
which proves our claim. Remember that this also guarantees the existence of $\lambda_{\pm}(B,q)$ as mention in Section \ref{sec.Sem}.

It is easy to see that $\nu=\iota_{*}\mu$. Using this we have
\begin{align*}
\lambda_+ (B,q) &= \lim_n\frac{1}{n}\int_{N}\log\|B^n(x)\|d\nu
              %\lim_n\frac{1}{n}\int_{M}\log\|B^n(i(x))\|d\mu\\
              = \lim_n\frac{1}{n}\int_{M}\log\|A^{5n}(x)\|d\mu\\
              &= 5\lambda_+(A,\tq)
              = 5p_3\log 2>0.
\end{align*}

The task is now to construct the sequence $(q_n)_n$. In order to do this, for each $n\in\N$ consider $w_n=(2n,2n+2,2n+1,2n-1,2n-1)$ and define
\[
 \beta(w_n)=(\alpha(2n)R_{\epsilon},\alpha(2n+2),\alpha(2n+1)R_{\delta},\alpha(2n-1),\alpha(2n-1)R_{\epsilon}),
\]
where $\epsilon=n^{-1}(n+1)^{-1}$, $\delta=\arctan(\epsilon)$ and,
\[
 R_{\epsilon} = \begin{pmatrix}
                  1 & 0\\
                  \epsilon & 1
                \end{pmatrix},\quad
R_{\delta} = \begin{pmatrix}
                  \cos(\delta) & -\sin(\delta)\\
                  \sin(\delta) & \cos(\delta)
                \end{pmatrix}.
\]

We proceed to define the sequence by
\[
 q_n=\sum_{w\neq w_n}p_{w}\delta_{\alpha(w)}+p_{w_n}\delta_{\beta(w_n)}.
\]
We claim that $W(q_n,q)\to 0$ if $n$ goes to infinite. Our proof starts with the observation that
\[
 \pi_n=\sum_{w\neq w_n}p_{w}\delta_{(\alpha(w),\alpha(w))}+p_{w_n}\delta_{(\alpha(w_n),\beta(w_n))}
\]
is a coupling of $q$ and $q_n$. Then,
$$
W(q_n,q) \leq \int d_{\infty}(u,v)\pi_k(u,v) = p_{w_n}d(\alpha(w_n),\beta(w_n))
$$
and the latter is bounded from above by
$$
\begin{aligned}
\max\{\|\alpha(2n)(1-R_{\epsilon})\|,\|\alpha(2n-1)(1-R_{\delta})\|,\|\alpha(2n+1)(1-R_{\epsilon})\|\}
\leq \epsilon (n+1) = \frac{1}{n}
\end{aligned}
$$
which proves our claim.

What is left is to show that $\lambda_+(B,q_n)=0$ for all $n$. For this we need the following lema.

\begin{lemma}\label{l.HV}
Let $H_x=\R(1,0)$ and $V_x=\R(0,1)$. If $Z_n=[0:\beta(w_n)]$ then, for all $x\in Z_n$ we have $B(x)H_x=V_{g(x)}$ and $B(x)V_x=H_{g(x)}$
\end{lemma}

\begin{proof}
Notice that for any $x\in Z_n$
\[
  B(x)=\begin{pmatrix}
       0 &-\epsilon^{-2}\sin(\delta) \\
       \epsilon^2\sin(\delta)+\epsilon\cos(\delta) & 0
       \end{pmatrix}.
\]
Which completes the proof.
\end{proof}
The rest of the proof follows the same arguments as the Bocker-Viana example. For more details we refer the reader to (\cite[Section~7.1]{BockerViana}).

\subsection{Discontinuity example in $\GL^2$}\label{sec.DiscGL2}

Let $M=(\GL)^{\Z}$ let $f:M\to M$  be the shift map over $M$ and $A:M\to\GL$ the product of random matrices. Now consider $X=\GL^2$ with the maximum norm, and let $N=X^{\Z}$ be the space of sequences over $X$ and $g:N\to N$ the shift map over $N$. As before, we can identify $N$ with $M$ using the function $\iota: M\to N$ defined by $\iota((\alpha_n)_n)=(\beta_n)_n$ where $\beta_n=(\alpha_{2n},\alpha_{2n+1})$ which is a bijection between $N$ and $M$.

With the above definition we can see that
\[
 g(\iota((\alpha_n)_n))=f^2((\alpha_n)_n)
\]
and defined $B:N\to \GL$ the linear cocycle induced by $A$ in $N$ through $B(\iota((\alpha_n)_n))=A^2((\alpha_n)_n)$.
In a similar way as in the previous example, there exist a measure $p$ and a sequence of measures $(p_k)_k$ on $X$ converging to $p$ in the Wasserstein topology, such that $\lambda_+(A,p_k)\nrightarrow\lambda_+(A,p)$.

Indeed, let $\alpha:\N\to\GL$ be defined by
$$
\alpha(2k-1) = \begin{pmatrix}
                  k & 0\\
                  0 & k^{-2}\\
                 \end{pmatrix}
\quand
\alpha(2k) = \begin{pmatrix}
                  k^{-2} & 0\\
                  0 & k\\
               \end{pmatrix}.
$$

Take the weights $p_k$ as in the previous section and let $\tq=\sum_{k\in \N}p_k\delta_{\alpha(k)},$ . Consider the space $\Omega=\N^2$ and define the measure on $X$ by $q=\sum_{w\in\Omega}p_{w}\delta_{\alpha(w)}$, where $\alpha(w)=(\alpha(w_1),\alpha(w_2))$ and, $p_w=p_{w_1} p_{w_2}$ if $w=(w_1,w_2)$. Let $\nu=q^{\Z}$ a measure on $N$.

Analysis similar to that in Section \ref{sec.DiscSL5} shows that $q\in P_1(X)$, and using that $\nu=i_{*}\mu$ we have
\begin{align*}
\lambda_+ (B,q) &= \lim_n\frac{1}{n}\int_N\log\|B^n(x)\|d\nu
              = \lim_n\frac{1}{n}\int_{M}\log\|A^{2n}(x)\|d\mu\\
              &= 2\lambda_+(A,\tq)
              = 2p_3\log 2>0.
\end{align*}
For each $n\in\N$ consider $w_n=(2n,2n-1)$ and define $\beta(w_n)=(\beta(2n),\beta(2n-1)),$ by
%\begin{align*}
\[\beta(2n) = \begin{pmatrix}
            1 & -\delta\\
            0 & 1
          \end{pmatrix}
          \alpha(2n)
          \begin{pmatrix}
            1 & 0\\
            \epsilon & 1
          \end{pmatrix}=
          \begin{pmatrix}
            0 & -n\delta\\
            \epsilon n & n
          \end{pmatrix},\\
\]
\[\beta(2n-1) = \begin{pmatrix}
            1 & 0\\
            \epsilon & 1
          \end{pmatrix}
          \alpha(2n-1)=
          \begin{pmatrix}
            n & 0\\
            \epsilon n & n^{-1}
          \end{pmatrix},\]
%\end{align*}
where $\delta=n^{-(1+\gamma)}$ with $0<\gamma<1$, $\epsilon=n^{-3}\delta^{-1}=n^{\gamma-2}$.

We proceed to define the sequence by
\[
 q_n=\sum_{w\neq w_n}p_{w}\delta_{\alpha(w)}+p_{w_n}\delta_{\beta(w_n)}.
\]
To prove that $W(q_n,q)\to 0$ if $n$ goes to infinite we consider the diagonal coupling of $q_n$ and $q$
\[
 \pi_n=\sum_{w\neq w_n}p_{w}\delta_{(\alpha(w),\alpha(w))}+p_{w_n}\delta_{(\alpha(w_n),\beta(w_n))}
\]
Hence, we have
\begin{align*}
W(q_n,q) & \leq \int d_{\infty}(u,v)\pi_n(u,v)
         = p_{w_n}d_{\infty}(\alpha(w_n),\beta(w_n))\\
         & < \max\{\|\beta(2n)-\alpha(2n)\|,\|\beta(2n-1)-\alpha(2n-1)\|\}\\
         & \leq \max\{\epsilon\sigma_{n},n^{-2}+n\delta\}
         = \max\{n^{\gamma-1},n^{-2}+n^{-\gamma}\}
         \leq 2n^{-l}
\end{align*}
where $l=\min\{\gamma,1-\gamma\}>0$, which proves our claim.

The rest of the proof, that is proving that $\lambda_+(B ,q_n)=0$ for all $n$, runs as before by noticing that for any $x\in Z_n=[0:\beta(w_n)]$
\[
 B(x)=\begin{pmatrix}
        0 &-n^{2}\delta\\
        \epsilon n^{-1} & 0
      \end{pmatrix}.
\]
Indeed, this guarantees that $B(x)H_x=V_{g(x)}$ and $B(x)V_x=H_{g(x)}$ where $H_x=\R(1,0)$ and $V_x=\R(0,1)$.
Finally, applying the first return map argument in \cite[Section~7.1]{BockerViana}, we conclude our proof.

\section*{Acknowledgements}

The first author thanks the Math Department of ICMC where the major part of the work was developed. This work is supported by FAPESP (Fundação de Amparo à Pesquisa do Estado de São Paulo), grant 2018/18990-0.
%------------------------------------------------------------------------------------------------------
%Bybliography

\bibliography{bibliography}
\bibliographystyle{plain}

\end{document}